\documentclass[12pt]{amsart}
\usepackage{latexsym,amsthm,amssymb,amsfonts}
\usepackage{hyperref}

\newtheorem{theorem}{Theorem}

\newtheorem{proposition}[theorem]{Proposition}

\theoremstyle{definition}

\theoremstyle{remark}

\newcommand{\PSL}{{\mathrm {PSL}}}

\newcommand{\Ker}{\operatorname{Ker}}
\newcommand{\Aut}{{\mathrm {Aut}}}

\newcommand{\Irr}{{\mathrm {Irr}}}

\newcommand{\St}{{\mathrm {St}}}

\newcommand{\bC}{{\mathbf{C}}}
\newcommand{\bN}{\mathbf{N}}

\newcommand{\Sy}{\textup{\textsf{S}}}
\newcommand{\rat}{{\mathrm {rat}}}
\newcommand{\ncf}{{\mathrm {ncf}}}
\newcommand{\Oinfty}{{\mathcal{O}_\infty}}

\begin{document}

\title[The character degree ratio]
{The character degree ratio and composition factors of a finite
group}

\thanks{Hung N. Nguyen was partially supported by the NSA Young
  Investigator Grant \#H98230-14-1-0293 and a Faculty Scholarship
  Award from Buchtel College of Arts and Sciences, The University
  of Akron}

\author{Mark L. Lewis}
\address{Department of Mathematical Sciences, Kent State University, Kent,
OH 44242, USA} \email{lewis@math.kent.edu}

\author[Hung Ngoc Nguyen]{Hung Ngoc Nguyen}
\address{Department of Mathematics, The University of Akron, Akron,
Ohio 44325, USA} \email{hungnguyen@uakron.edu}



\date{\today}

\begin{abstract} For a finite non-abelian group $G$ let $\rat(G)$ denote the largest ratio
of degrees of two nonlinear irreducible characters of $G$. We prove
that the number of non-abelian composition factors of $G$ is bounded
above by $1.8\ln(\rat(G))+1.3$.
\end{abstract}

\maketitle


\section{Introduction}

For a finite non-abelian group $G$ let $\rat(G)$ be the largest
ratio of degrees of two nonlinear irreducible characters of $G$.
That is,
\[\rat(G)=\frac{b(G)}{c(G)},\] where $b(G)$ denotes the largest degree of an (ordinary) irreducible
character of $G$ and $c(G)$ denotes the minimum degree of a
nonlinear irreducible character of $G$. This ratio is often referred
to as the \emph{character degree ratio} of $G$. When $G$ is abelian,
we adopt a convention that $c(G)=1$ and $\rat(G)=1$.

The main result of this note shows that the number of non-abelian
composition factors of an arbitrary finite group is logarithmically
bounded above by its character degree ratio.

\begin{theorem}\label{theorem-main} The number of non-abelian composition factors of a
finite group $G$ is bounded above by $1.8\ln(\rat(G))+1.3$.
\end{theorem}

In \cite[Theorem C]{Isaacs2}, I.\,M.~Isaacs has shown that if $G$ is
solvable then the derived length of $G$ is bounded by
$3+4\log_2(\rat(G))$. Our Theorem~\ref{theorem-main} can be
considered as a \emph{nonsolvable} version of Isaacs's result.
Theorem~\ref{theorem-main} also significantly improves a result of
J.\,P.~Cossey and the second author \cite[Theorem~A]{Cossey-Nguyen}
that if $S$ a non-abelian composition factor of $G$ different from
the simple linear groups $\PSL_2(q)$, then the number of times that
$S$ occurs as a composition factor of $G$ is bounded in terms of
$\rat(G)$. We refer the reader to
\cite{Cossey-Nguyen,Lewis,Maroti-Nguyen,Moreto-Nguyen} for more
discussion on the influence of the character degree ratio and
character degrees in general on the structure of finite groups.

Let us now describe some ideas in the proof of the main result. Let
$\Oinfty(G)$ denote the solvable radical of $G$, i.e. $\Oinfty(G)$
is the maximal solvable normal subgroup of $G$. Then
$\Oinfty(G/\Oinfty(G))=1$ and the number of non-abelian composition
factors of $G$ equals to that of $G/\Oinfty(G)$. Moreover, as
$c(G/\Oinfty(G))\geq c(G)$ and $b(G/\Oinfty(G))\leq b(G)$, we have
$\rat(G/\Oinfty(G))\leq \rat(G)$. Therefore in the proof of
Theorem~\ref{theorem-main} we can assume that $\Oinfty(G)=1$. It
follows that if $N$ be a minimal normal subgroup of $G$, then $N$ is
isomorphic to a direct product of copies of a non-abelian simple
group $S$.

One of the key results we need is \cite[Theorem~1]{Cossey-Nguyen},
which asserts that if $S$ is a nonabelian simple group not
isomorphic to $\PSL_2(q)$, then $S$ has two non-principal
irreducible characters $\alpha$ and $\beta$ which are extendible to
$\Aut(S)$ such that $\alpha(1)/\beta(1)>|S|^{1/14}.$ This allows us
to produce two character degrees of $G$ of large ratio, and to
obtain a better bound $\ncf(G)<1.8\ln(\rat(G))$ in the case
$S\ncong\PSL_2(q)$. At this point and from now on, we write
$\ncf(G)$ to denote the number of non-abelian composition factors of
$G$.

It turns out that the case $S\cong\PSL_2(q)$ is exceptional and
raises some complications. We first show that the number of copies
of $S$ in $N$ is bounded above by $\ln(\rat(G))+1.3$ and this solves
Theorem~\ref{theorem-main} when $G/N$ is solvable. When $G/N$ is
nonsolvable, we use another key result of J.\,P.~Cossey, Z.~Halasi,
A.~Mar\'{o}ti, and H.\,N.~Nguyen
\cite[Theorem~6]{Cossey-Halasi-Maroti-Nguyen} on upper bound for the
product of the orders of the non-abelian composition factors of a
finite group $G$ in terms of its largest character degree.

To end this introduction, we would like to make a couple of remarks.
Firstly, though the bound obtained is of the right order of
magnitude as shown by characteristically simple groups, we think
that the constants $1.8$ and $1.3$ can be improved and it would be
interesting to find the correct bounding constants. Indeed, we know
of no finite groups with $\ncf(G)\geq \ln(\rat(G))+1$. Secondly,
since the available proofs of \cite[Theorem~1]{Cossey-Nguyen} and
\cite[Theorem~6]{Cossey-Halasi-Maroti-Nguyen} both depend on the
classification of finite simple groups, Theorem~\ref{theorem-main}
depends on the classification as well.


\section{The simple linear groups $\PSL_2(q)$}

We start with the following result, which implies
Theorem~\ref{theorem-main} in the case $S\cong \PSL_2(q)$ and $G/N$
is solvable.

\begin{proposition}\label{proposition-main} Let $S=\PSL_2(q)$ where $q\geq 5$ is a prime power.
Assume that $N:=S\times\cdots\times S$, a direct product of $k$
copies of $S$, is a minimal normal subgroup of $G$. Then
$k<\ln(\rat(G))+1.3$.
\end{proposition}

\begin{proof} Write $N = S_1 \times \cdots \times S_k$ where $S_i\cong S$ for every $i=1,2,...,k$. Let $T = \bN_G (S_1)$, so $|G:T| =
k$. Furthermore $S_1$ can be considered as a subgroup of
$T/\bC_G(S_1)$, which in turn is isomorphic to a subgroup of
$\Aut(S_1)$. Consider the so-called Steinberg character $\St_{S_1}$
of degree $q$ of $S_1$. It is well-known that $\St_{S_1}$ is
extendible to $\Aut(S_1)$ (see \cite{Feit} for instance). Therefore
$\St_{S_1}$ is extended to an irreducible character, say $\psi$, of
$T/\bC_G(S_1)$. It follows that $\St_{S_1}$ is extended to an
irreducible character, say $\chi$, of $T$ whose kernel contains
$\bC_G(S_1)$. Now since $S_2 \times...\times S_k$ is inside
$\Ker(\chi)$, we conclude that the character $\St_{S_1}\times
1_{S_2}\times....\times 1_{S_k}\in\Irr(N)$ is extendible to $T$.

Observe that the stabilizer of $\St_{S_1}\times
1_{S_2}\times....\times 1_{S_k}$ normalizes $S_1$, and
$\St_{S_1}\times 1_{S_2}\times....\times 1_{S_k}$ has $k$ conjugates
under the action of $G$. Thus, $T$ must be the stabilizer of
$\St_{S_1}\times 1_{S_2}\times....\times 1_{S_k}$ in $G$. Applying
Clifford's theorem, we deduce that $\St_{S_1} (1) k=kq$ is a
character degree of $G$.  In particular, $c(G) \leq q k$.

We now show that $b(G)\geq q^k$.  Since $\Aut (S)$ stabilizes
$\St_S$, the product character $\St_S \times\dots \times
\St_S\in\Irr(N)$ is invariant under $\Aut(N) = \Aut(S) \wr \Sy_k$.
Using \cite[Lemma~1.3]{Mattarei} (see also
\cite[Lemma~5]{Bianchi-Lewis}), we deduce that $\St_S \times \dots
\times \St_S$ is extendible to $\Aut(N)$. Since $N$ can be embedded
into $G/\bC_G(N)$ and $G/\bC_G(N)$ embeds into $\Aut(N)$, it follows
that $\St_S\times\dots \times \St_S$ is extendible to $G/\bC_G(N)$
and hence $b(G)\geq \St_S(1)^k=q^k$.

Combining the bounds for $b(G)$ and $c(G)$ together, we obtain
$$
\rat(G)= \frac{b(G)}{c(G)}\geq \frac{q^k}{kq}.
$$
As it is easy to see that $k < 1.445^k$ for all positive integers
$k$, it follows that $\rat (G) > q^k/(q 1.445^k) = q^{k-1}/1.445^k.$
Therefore we obtain $\ln (\rat (G)) > (k-1) \ln q - k (\ln (1.445)$,
which yields
$$
k < \frac {\ln (\rat (G))}{\ln q - \ln 1.445} + \frac {\ln q}{\ln q
- \ln 1.445}.
$$
It is straightforward to check that $\ln q - \ln 1.445 >1$ and $\ln
q/(\ln q-\ln 1.445)<1.3$ for every $q\geq 5$. We now obtain the
required bound $k < \ln (\rat (G)) + 1.3$.
\end{proof}

As mentioned already, Proposition~\ref{proposition-main} implies the
main result in the case $S\cong \PSL_2(q)$ and $G/N$ is solvable. By
making use of the following deep result, we can prove the bound
$\ncf(G)<1.4\ln(\rat(G))$ whenever $S=\PSL_2(q)$ and $G/N$ is
non-abelian.

\begin{theorem}[Theorem 6 of \cite{Cossey-Halasi-Maroti-Nguyen}]\label{theorem-Cossey...} Let
  $G$ be a finite group. Then the product of the orders of the
  non-abelian composition factors of $G$ is at most $b(G)^3$.
\end{theorem}

\begin{proposition}\label{proposition-main2} Let $S=\PSL_2(q)$ where $q\geq 5$ is a prime power.
Assume that $N:=S\times\cdots\times S$, a direct product of $k$
copies of $S$, is a minimal normal subgroup of $G$ such that $G/N$
is non-abelian. Then $\ncf(G)<1.4\ln(\rat(G))$.
\end{proposition}

\begin{proof} As in the proof of Proposition~\ref{proposition-main}, the character
$\St_S\times\dots\times \St_S\in\Irr(N)$ is extended to an
irreducible character, say $\chi$, of $G/N$, which can also be
viewed as an irreducible character of $G$. By Gallagher's Theorem
\cite[Corollary~6.17]{Isaacs1}, there is a bijection
$\lambda\leftrightarrow \lambda \chi$ between $\Irr(G/N)$ and the
set of irreducible characters of $G$ lying above
$\St_S\times\dots\times \St_S$. In particular, by taking $\lambda$
to be an irreducible character of $G/N$ of the largest degree, we
deduce that
\[b(G/N)\chi(1)=b(G/N)\St_S(1)^k=b(G/N)q^k\] is an irreducible character
degree of $G$. As it is clear that $c(G/N)\geq c(G)$ and note that
$G/N$ is non-abelian, it follows that
\[\rat(G)\geq \frac{b(G/N)q^k}{c(G/N)}.\] In particular,
\begin{equation}\label{equation1}k\leq \frac{\ln (\rat(G))}{\ln q}.\end{equation}

We now aim to bound $\ncf(G)-k$ in terms of $\rat(G)$. Since $q^k$
is a character degree of $G$, we have $c(G)\leq q^k$ and hence
\[\rat(G)\geq \frac{b(G/N)q^k}{q^k}=b(G/N).\]
Applying Theorem~\ref{theorem-Cossey...}, we deduce that $\rat(G)^3$
is at least the product of the orders of the non-abelian composition
factors of $G/N$.

Note that the number of non-abelian composition factors of $G/N$ is
$\ncf(G)-k$ and the order of each non-abelian composition factor is
at least $60$. It follows that
\[\rat(G)\geq 60^{(\ncf(G)-k)/3},\] which yields
\begin{equation}\label{equation2}\ncf(G)-k\leq \frac{3\ln(\rat(G))}{\ln
60}.\end{equation}

Inequalities \ref{equation1} and \ref{equation2} then imply that
\[\ncf(G)\leq \frac{\ln (\rat(G))}{\ln q}+\frac{3\ln(\rat(G))}{\ln
60},\] and thus the proposition follows.
\end{proof}

We have proved Theorem~\ref{theorem-main} in the case $S\cong
\PSL_2(q)$.


\section{Other non-abelian simple groups}

We now handle the case $S\ncong \PSL_2(q)$.

\begin{proposition}\label{proposition-main3} Let $S$ be a non-abelian simple group different from $\PSL_2(q)$.
Assume that $N:=S\times\cdots\times S$ is a minimal normal subgroup
of $G$. Then $\ncf(G)<1.8\ln(\rat(G))$.
\end{proposition}

Our current proof of Proposition~\ref{proposition-main3} is largely
based on Theorem~\ref{theorem-Cossey...} and the following result,
whose proofs both depend on the classification of finite simple
groups.

\begin{theorem}[Theorem 1 of \cite{Cossey-Nguyen}]\label{theorem-extending-character} Let $S$ be a non-abelian simple
group different from $\PSL_2(q)$. Then $S$ has two non-principal
irreducible characters $\alpha$ and $\beta$ which are extendible to
$\Aut(S)$ such that
\[\frac{\alpha(1)}{\beta(1)}>|S|^{1/14}.\]
\end{theorem}

\begin{proof}[Proof of Proposition~\ref{proposition-main3}] By Theorem~\ref{theorem-extending-character}, the simple group $S$
has two non-principal irreducible characters $\alpha$ and $\beta$
which are extendible to $\Aut(S)$ such that
$$\frac{\alpha(1)}{\beta(1)}>|S|^{1/14}.$$ The product
characters $\alpha_1:=\alpha \times \cdots \times \alpha$ and
$\beta_1:=\beta \times \cdots \times \beta$ are then invariant under
$\Aut(N)=\Aut(S)\wr \Sy_k$. Arguing as before, we see that
$\alpha(1)^k$ and $\beta(1)^k$ are irreducible character degrees of
$G$.

As in the proof of Proposition~\ref{proposition-main2}, one can show
that $b(G/N)\alpha(1)^k$ is an irreducible character degree of $G$.
Together with the conclusion of the last paragraph, we have
\[\rat(G)\geq
\frac{b(G/N)\alpha(1)^k}{\beta(1)^k}.\] As
$\alpha(1)/\beta(1)>|S|^{1/14}$, we obtain
\[\rat(G)>b(G/N)|S|^{k/14},\] and hence
\[\rat(G)/|S|^{k/14}>b(G/N).\]
Applying Theorem~\ref{theorem-Cossey...} again, we then have that
$\rat(G)^3/|S|^{3k/14}$ is at least the product of the orders of
non-abelian composition factors of $G/N$.

Since the number of non-abelian composition factors of $G/N$ is
precisely equal to $\ncf(G)-\ncf(N)=\ncf(G)-k$, it follows that
\[\frac{\rat(G)^3}{|S|^{3k/14}}>{60}^{\ncf(G)-k}.\] Note that the order of any
$S\neq \PSL_2(q)$ is at least $2\,520$. We deduce that
\[\rat(G)> 2\,520^{\ncf(G)/14},\] which in turns implies that
\[\ncf(G)<1.8\ln(\rat(G)),\] as claimed.

\end{proof}


\section{Proof of the main result}

\begin{proof}[Proof of Theorem~\ref{theorem-main}] As mentioned in the Introduction, to prove the theorem it suffices to assume that $\Oinfty(G)=1$.
Let $N$ be a minimal normal subgroup of $G$. As $G$ has trivial
solvable radical, $N$ is direct product of copies of a non-abelian
simple group $S$.

When $S=\PSL_2(q)$ and $G/N$ is solvable we have
$\ncf(G)<\ln(\rat(G))+1.3$ by Proposition~\ref{proposition-main}.
Otherwise we have the bound $\ncf(G)<1.8\ln(\rat(G))$ by
Propositions~\ref{proposition-main2} and~\ref{proposition-main3}.
The theorem is now completely proved.
\end{proof}

\section*{Acknowledgement} The authors are grateful to Alexandre Turull for
raising a question on the character degree ratio that leads to the
main result of this paper.



\begin{thebibliography}{13}

\bibitem{Bianchi-Lewis}
M. Bianchi, D. Chillag, M.\,L. Lewis, and E. Pacifici, Character
degree graphs that are complete graphs, {\it Proc. Amer. Math. Soc.}
{\bf135} (2007), 671-676.

\bibitem{Cossey-Halasi-Maroti-Nguyen}
J.\,P. Cossey, Z. Halasi, A. Mar\'{o}ti, and H.\,N. Nguyen, On a
conjecture of Gluck, \emph{Math. Z.}, to appear. ISSN (Online)
1432-1823, ISSN (Print) 0025-5874, DOI 10.1007/s00209-014-1403-6


\bibitem{Cossey-Nguyen}
J.\,P. Cossey and H.\,N. Nguyen, Controlling composition factors of
a finite group by its character degree ratio, \emph{J. Algebra}
\textbf{403} (2014), 185-200.


\bibitem{Feit}
W. Feit, Extending Steinberg characters, linear algebraic groups and
their representations, {\it Contemp. Math.} {\bf 153} (1993), 1-9.


\bibitem{Isaacs1}
  I.\,M. Isaacs, {\it Character theory of finite groups}, AMS Chelsea
Publishing, Providence, Rhode Island, 2006.


\bibitem{Isaacs2}
I.\,M. Isaacs, Character kernels and degree ratios in finite groups,
{\it J. Algebra} {\bf322} (2009), 2220-2234.


\bibitem{Lewis}
 M.\,L. Lewis, An overview of graphs associated with character degrees and conjugacy
 class sizes in finite groups, \emph{Rocky Mountain J. Math.} \textbf{38} (2008), 175-211.

\bibitem{Maroti-Nguyen}
A. Mar\'{o}ti and H.\,N. Nguyen, Character degree sums of finite
groups, \emph{Forum Math.}, to appear. ISSN (Online) 1435-5337, ISSN
(Print) 0933-7741, DOI: 10.1515/forum-2013-0066, October 2013

\bibitem{Mattarei}
S. Mattarei, On character tables of wreath products, \emph{J.
Algebra}~\textbf{175} (1995), 157-178.

\bibitem{Moreto-Nguyen}
A. Moret\'{o} and H.\,N. Nguyen, On the average character degree of
finite groups, \emph{Bull. Lond. Math. Soc.} \textbf{46} (2014),
454-462.

\end{thebibliography}
\end{document}